\documentclass[a4paper, 12pt]{article}

\usepackage{amsfonts}
\usepackage{amssymb}
\usepackage{amsthm}
\usepackage{amssymb}
\usepackage[all]{xy}
\usepackage{graphicx}
\usepackage{amsmath, amssymb}
\usepackage{gensymb}

\usepackage{hyperref}
\usepackage{caption}
\usepackage{subcaption}

\newcommand{\poq}{\frac{p}{q}}
\newcommand{\andd}{\text{ and }}

\theoremstyle{thmstyleone}
\newtheorem{theorem}{Theorem}

\newtheorem{lemma}[theorem]{Lemma}

\newtheorem{definition}{Definition}

\usepackage{authblk}

\begin{document}

\title{No-Three-in-a-$\Theta:$ Variations on the No-Three-in-a-Line Problem}

\author[1]{Natalie Dodson}

\author[2]{Anant Godbole}

\author[3]{Dashleen Gonzalez}

\author[4]{Ryan Lynch}

\author[5]{Lani Southern}

\affil[1]{Middlebury College}
\affil[2]{East Tennessee State University}
\affil[3]{University of Puerto Rico Mayag\"uez}
\affil[4]{University of Minnesota, Twin Cities}
\affil[5]{Willamette University}
\date{}

\maketitle


\begin{abstract}
We pose a natural generalization to the well-studied and difficult no-three-in-a-line problem: How many points can be chosen on an $n \times n$  grid such that no three of them form an angle of $\theta$? In this paper, we classify which angles yield nontrivial problems, noting that some angles appear in surprising configurations on the grid. We prove a lower bound of $2n$ points for angles $\theta$ such that $135 \degree \leq \theta < 180 \degree$, and further explore the case $\theta = 135\degree$, utilizing geometric properties of the grid to prove an upper bound of $3n - 2$ points. Lastly, we generalize the proof strategy used in proving the upper bound for $\theta = 135\degree$ to provide a general upper bound for all angles.
\end{abstract}

\maketitle

\section{Introduction}\label{secIntro}

The no-three-in-a-line problem is an open problem first proposed by Henry Dudeney in the early 20th century. It asks, what is the maximum number of points that can be placed on an $n \times n$ grid such that no three of them lie in a line? \cite{HALL1975336} This problem is equivalent to the $\theta = 180\degree$ case. Directly applying the pigeonhole principle on the rows of the grid gives an upper bound of $2n$ for the no-three-in-a-line-problem. However, it is unknown whether this upper bound is tight. This uncertainty is not due to lack of effort from many mathematicians. Hall, Jackson, Sudbery, and Wild proved a lower bound of $(\frac{3}{2}- \epsilon)n$ for all $n$ \cite{HALL1975336}. Guy and Kelly also provided a probabilistic conjecture for an upper bound of about $1.874n$, which was later corrected to $1.814n$ by Gabor Ellmann
\cite{guy_kelly_1968}.

Gossell and Johnson solved a related problem, determining that the maximum number of points that lie on an $n \times n$ grid such that no three of them lie in an ``L," or an angle of $90\degree$ is $2n-2$ \cite{GossellJohnson}. An alternate proof for the ``L" case was also presented by Godbole and LaFollette \cite{godbole_lafollette_2019}.

We ask and begin to answer a generalization of these problems: How many points can be placed on an $n\times n$ grid such that no three of them form an angle of $\theta$ ? In Section \ref{secInterestingAngles}, we classify which choices of $\theta$ yield interesting problems. In Section \ref{secLowerBound}, we prove a lower bound for interesting angles between $135 \degree$ and $180 \degree$. In Section \ref{secUpperBound135}, we find an upper bound for the $135\degree$ case, relying on geometric properties specific to $\theta = 135\degree$. We generalize to all interesting angles in Section \ref{secGeneralUpperBound} and obtain upper bounds for all angles that can appear on the grid.

\section{Interesting Angles}\label{secInterestingAngles}

In this section, we determine which choices of $\theta$ yield interesting problems. We begin with the following definitions:

\begin{definition}
    A \emph{grid} $G_n$ is a set of $n^2$ points arranged with $n$ points in each row and column for some natural number $n$. We denote points in $G_n$ by their corresponding integer coordinates in $\mathbb{R}^2,$ starting with $(1, 1)$ for the bottom left hand corner.  
\end{definition}

\begin{definition}
    A \emph{construction} $C_n$ is a subset of points of $G_n$. 
\end{definition}

\begin{definition}
    Every point $P \in C_n$ is called a \emph{chosen point}.  
\end{definition}

\begin{definition}
An angle $\theta$ is \emph{interesting} if there exists a construction $C_n$ for some $n \in \mathbb{N}$ such that for some $A, B,$ and $ C \in C_n$, one of the interior angles of triangle $ABC$  is equal to $\theta.$ 
\end{definition}

\begin{theorem}
\label{interestingIFFrational}
An angle $\theta$ is interesting if and only if $\tan(\theta)$ is rational.
\end{theorem}

\begin{proof}

\begin{figure}[ht]
    \centering
    \includegraphics[width=4.5cm]{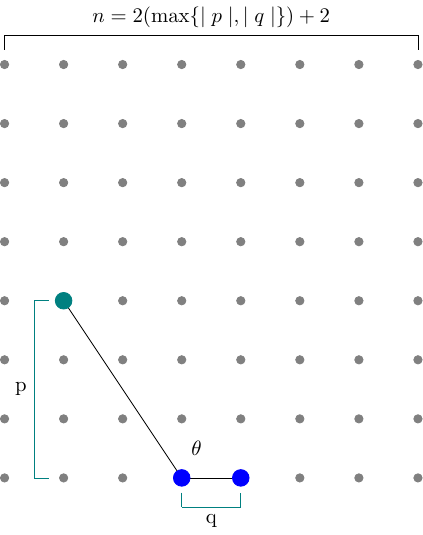}
    \caption{A construction of $\theta$ for $\theta=\arctan(\frac{-3}{2}).$}
    \label{fig:construction}
\end{figure}

Let $\theta$ be an arbitrary angle such that $\tan(\theta)=\frac{p}{q}$ for some integers $p$ and $q.$ If $\frac{p}{q}<0,$ choose $q$ to be negative. Let $n=2 (\text{max}\{ \lvert P \rvert, \lvert q \rvert \})+2.$ Take the points $(\frac{n}{2}+1,1),$ $(\frac{n}{2}+2,1),$ and $(\frac{n}{2}+1+q,1+p)$ in $C_n.$ These points form an angle of $\theta,$ as shown in Figure \ref{fig:construction}, so $\theta$ is interesting.

Let $\theta$ be an arbitrary interesting angle. Then for some $C_n$, there exists $A = (A_x, A_y), B = (B_x, B_y), C = (C_x, C_y) \in C_n$ such that $\angle ABC = \theta$. 

The magnitude of the cross product is given by $\| \overrightarrow{BA} \times \overrightarrow{BC} \| = \| \overrightarrow{BA} \| \|\overrightarrow{BC} \| \sin{\theta}.$ The cosine of the angle is given by $\cos{\theta}= \frac{\overrightarrow{BA} \cdot \overrightarrow{BC}}{\| \overrightarrow{BA} \| \|\overrightarrow{BC} \|}.$ Thus

\begin{align}
    \tan{\theta} & =\frac{\sin{\theta}}{\cos{\theta}} \nonumber\\
    &= \left( \frac{\| \overrightarrow{BA} \times \overrightarrow{BC} \|}{\| \overrightarrow{BA} \| \|\overrightarrow{BC} \|} \right) \left( \frac{\| \overrightarrow{BA} \| \|\overrightarrow{BC} \|}{\overrightarrow{BA} \cdot \overrightarrow{BC}} \right) \nonumber\\
    &=\frac{\| \overrightarrow{BA} \times \overrightarrow{BC} \|}{\overrightarrow{BA} \cdot \overrightarrow{BC}} \nonumber\\
    &= \frac{(A_x-B_x)(C_y-B_y)-(C_x-B_x)(A_y-B_y)}{(A_x-B_x)(C_x-B_x)+(A_y-B_y)(C_y-B_y)}, \nonumber 
\end{align}

 which is rational since $A,$ $B,$ and $C$ have integer coordinates. 
\end{proof}

\section{Lower Bound for \texorpdfstring{$135\degree \leq \theta < 180 \degree$}{135 degree leq theta < 180 degree}}\label{secLowerBound}

\begin{definition}
    A construction $C_n$ is \emph{peaceful} for $\theta$ if no three points in $C_n$ form an angle of $\theta$.
\end{definition}

\begin{theorem}
\label{lowerBound}
For $135 \degree \leq \theta < 180 \degree,$ there exists a peaceful $C_n$ with $\lvert C_n \rvert=2n$.
\end{theorem}

\begin{proof}

\begin{figure}
\centering
\begin{minipage}{.5\textwidth}
  \centering
  \includegraphics[height=4.5cm]{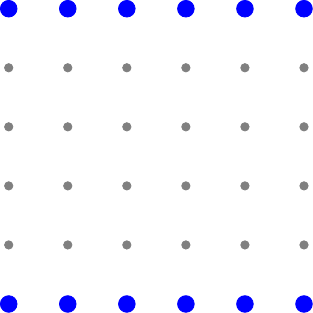}
  \captionof{figure}{Lower Bound Construction $135 \degree \leq \theta < 180 \degree$ for $G_6$}
  \label{fig: lower bound construction}
\end{minipage}%
\begin{minipage}{.5\textwidth}
  \centering
  \includegraphics[height=4.5cm]{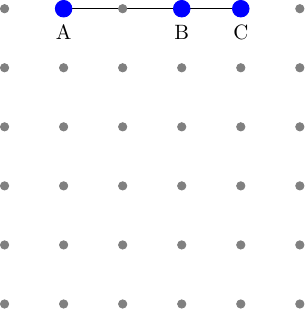} 
  \captionof{figure}{Case 1 of Theorem \ref{lowerBound}}
  \label{fig:lower bound case 1}
\end{minipage}
\end{figure}

Let $\theta$ be an arbitrary angle such that $135 \degree \leq \theta < 180 \degree.$ Let $n \in \mathbb{N}$ be arbitrary. Consider the construction $C_n = \{(x, y) \mid (x, y) \in G_n, y = 1 \text{ or } n\}$, as shown for $G_6$ in Figure \ref{fig: lower bound construction}. Let $A = (A_x, A_y), B = (B_x, B_y),$ and $C = (C_x, C_y) \in C_n$. There are three cases for $\angle {ABC}$: 

\begin{enumerate}
    \item $A_y = B_y = C_y$, as shown for $A_y = n$ in Figure \ref{fig:lower bound case 1}. In this case $\angle {ABC} = 180 \degree$, so $\angle {ABC} \neq \theta.$

    \item Without loss of generality, $A_y = C_y = n$ and $B_y = 1$. An example is shown in Figure \ref{fig: lower bound case 2}.
The largest angle is constructed if $A = (1,n)$, $B= (B_x,1)$ and $C = (n, n)$. By way of contradiction, suppose that $\angle ABC > 90 \degree.$ Then $ABC$ is an obtuse triangle and the side opposite the angle, $\overline{AC}$ must be the longest side of the the triangle. $\overline{AC}$ has length $n-1.$ However, by the Pythagorean theorem, $\overline{AB}$ has length $\sqrt{(n-1)^2+(B_x-1)^2} \geq n-1$ and $\overline{BC}$ has length $\sqrt{(n-1)^2+(n-B_x)^2} \geq n-1$, which is a contradiction. Therefore, $\angle ABC \not > 90 \degree$ and $\angle ABC \neq \theta.$

    \item   $A_y = 1 = B_y$ and $C_y = n.$ An example is shown in Figure \ref{fig: lower bound case 3}. Then the largest angle that can occur comes from taking $A=(1,1),$ $B=(B_x,1)$ and $C=(n,n)$ for some natural number $B_x$ such that $1<B_x\leq n.$ If $B_x = n,$ $\angle ABC=90 \degree \neq \theta.$ If $B_x<n,$ consider the right triangle formed by the points $(B_x,1),$ $(n,1),$ and $(n,n).$ Let $\theta_1$ be the angle inside the triangle with $(B_x,1)$ at its vertex, as shown in Figure \ref{fig: lower bound case 3}. Note that $\angle ABC=180 \degree - \theta_1.$ Since, $1<B_x<n,$ $\tan(\theta_1)=\frac{n-1}{n-B_x}>1$ so $\theta_1>45 \degree.$ Thus $\angle ABC < 180 \degree - 45 \degree = 135 \degree$ and $\angle ABC \neq \theta.$

\begin{figure}[h]
\centering
\begin{minipage}{.5\textwidth}
  \centering
  \includegraphics[height=4.5cm]{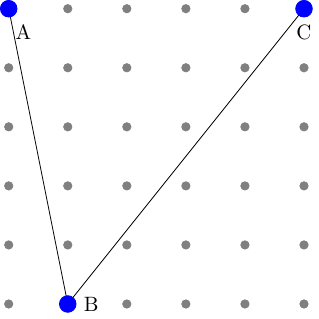}
  \captionof{figure}{Case 2 of Theorem \ref{lowerBound}}
  \label{fig: lower bound case 2}
\end{minipage}%
\begin{minipage}{.5\textwidth}
  \centering
  \includegraphics[height=4.5cm]{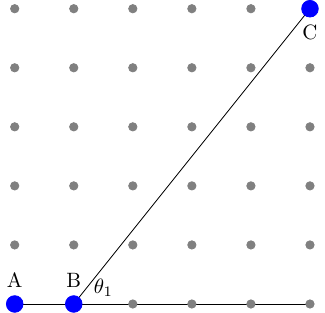} 
  \captionof{figure}{Case 3 of Theorem \ref{lowerBound}}
  \label{fig: lower bound case 3}
\end{minipage}
\end{figure}

\end{enumerate}
\end{proof}

\section{Upper Bound for \texorpdfstring{$\theta=135 \degree$}{theta=135 degree}} \label{secUpperBound135}

\begin{definition}

A \emph{$\frac{p}{q}$-slope-bucket} is a subset of points of $G_n$ that lie on a line of slope $\frac{p}{q}$ for some fixed $\frac{p}{q}.$ When the choice of $\frac{p}{q}$ is clear, we may refer to it simply as a $\emph{slope bucket}$.
\end{definition}

In this section, ``slope buckets" refer to $-1$-slope-buckets. The $-1$ slope buckets of $G_4$ are shown in Figure \ref{fig:buckets}.

\begin{figure}[ht]
    \centering
    \includegraphics[width=4.5cm]{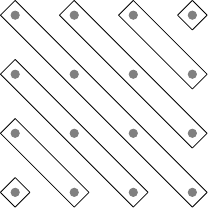}
    \caption{The $-1$-slope-buckets $G_4$}
    \label{fig:buckets}
\end{figure}

\begin{definition}
An \emph{interior point} $A = (A_x, A_y)$ is a point in a construction $C_n$ such that there exists distinct $B = (B_x, B_y)$ and $C =(C_x, C_y) \in C_n$ where either
\begin{itemize}
    \item $A_x = B_x = C_x$ and $B_y < A_y < C_y$, or
    \item $A_y = B_y = C_y$ and $B_x < A_x < C_x$, or
    \item $A, B,$ and $C$ are in the same slope bucket and $B_x < A_x < C_x$. 
\end{itemize}  
\end{definition}

\begin{figure}[h]
    \centering
\includegraphics[width=4.5cm]{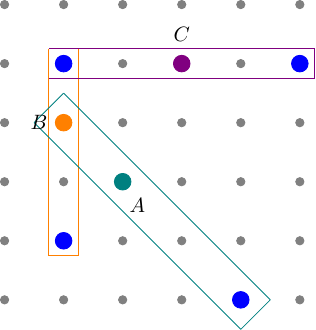}
    \caption {$A, B$ and $C$ are interior points with respect to its $-1$-slope bucket, its column, and its row, respectively.}
    \label{fig:interior points}
\end{figure}

\begin{lemma}
\label{same column}
Suppose there exist two chosen points $A,B$ with $A_y>B_y$ in the same column, $k$. The two cases enumerated below are shown in Figure \ref{fig: same column lemma}.
\begin{enumerate}
    \item If another point, $C,$ is chosen in the slope bucket of $A$ with $C_x<k,$ $\angle CAB=135 \degree.$
    \item If another point, $D$, is chosen in the slope bucket of $B$ with $D_x > k$, $\angle ABD = 135\degree$.
\end{enumerate}

\end{lemma}

\begin{figure}
    \centering
\includegraphics[width=4.5cm]{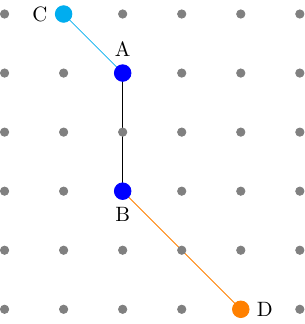}
    \caption{Lemma \ref{same column} $\angle CAB$ and $\angle ABD$}
    \label{fig: same column lemma}
\end{figure}

\begin{proof}~
\begin{enumerate}
    \item Since $A$ and $C$ are in the same $-1$-slope-bucket, $C_x-A_x=-(C_y-A_y)$. The vectors between points are given by $\vec{AC}= \langle C_x-A_x,C_y-A_y \rangle$ and $\vec{AB}=\langle B_x-A_x,B_y-A_y \rangle= \langle 0,B_y-A_y \rangle$. Then $\angle{CAB}$ is given by
    \begin{align*}
        \angle{CAB} 
        &=\arccos{\frac{\vec{AC} \cdot \vec{AB}}{\| \vec{AC} \| \| \vec{AB} \|}} \\
        &=\arccos{\frac{-1}{\sqrt{2}}}
        = 135 \degree.
    \end{align*}
    \item Since $B$ and $D$ are in the same $-1$-slope-bucket, $D_x-B_x=-(D_y-B_y)$. The vectors between the points are given by $\vec{BD}= \langle D_x-B_x,D_y-B_y \rangle $ and $\vec{BA}= \langle A_x-B_x,A_y-B_y \rangle = \langle 0,A_y-B_y \rangle.$ Then $\angle{ABD}$ is given by
    \begin{align*}
        \angle{ABD} 
        &=\arccos{\frac{\vec{BD} \cdot \vec{BA}}{\| \vec{BD} \| \| \vec{BA} \|}} \\
        &= \arccos{\frac{-1}{\sqrt{2}}}
        = 135 \degree.
    \end{align*}
\end{enumerate} 
\end{proof}

\begin{lemma}
\label{same row}
Suppose there exist two chosen points $A,B$ with $A_x>B_x$ in the same row, $k$. The two cases enumerated below are shown in Figure \ref{fig:same row lemma}.
\begin{enumerate}
    \item If another point, $C,$ is chosen in the slope bucket of $A$ with $C_y<k,$ $\angle CAB=135 \degree.$
    \item If another point, $D,$ is chosen in the slope bucket of $B$ with $D_y > k,$ $\angle DBA=135 \degree.$
\end{enumerate}
\end{lemma}

\begin{figure}[h]
    \centering
    \includegraphics[width=4.5cm]{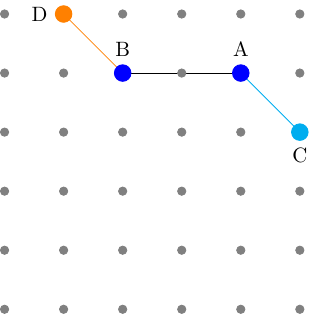}
    \caption{Lemma \ref{same row} $\angle CAB$ and $\angle DBA$}
    \label{fig:same row lemma}
\end{figure}

\begin{proof}
    These cases are analogous to the cases Lemma \ref{same column}.
\end{proof}

\begin{lemma}
\label{same bucket}
Suppose there exist two chosen points $A,B$ with $A_y>B_y$ in the same slope bucket, $k$. There are four cases:
\begin{enumerate}
    \item If another point, $C,$ is chosen in the column of $A$ with $A_y < C_y,$ $\angle CAB=135 \degree.$ 
    \item If another point, $D,$ is chosen in the column of $B$ with $ B_y > D_y,$ $\angle ABD=135 \degree.$
    \item If another point, $E,$ is chosen in the row of $A$ with $A_x > E_x,$ $\angle EAB=135 \degree$.
    \item If another point, $F,$ is chosen in the row of $B$ with $B_x < F_x,$ $\angle ABF=135 \degree.$
\end{enumerate}
The four cases below are shown in Figure \ref{fig:same bucket lemma}.
\end{lemma}

\begin{figure}
    \centering
    \includegraphics[width=4.5cm]{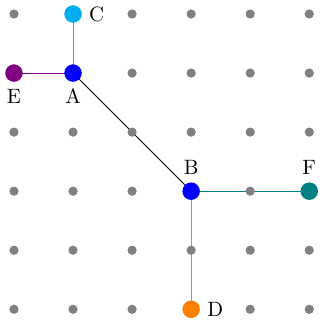}
    \caption{Lemma \ref{same bucket} $\angle CAB, \angle DBA, \angle EAB,$ and $\angle ABF$}
    \label{fig:same bucket lemma}
\end{figure}

\begin{proof}
\begin{enumerate}
    \item Consider a function from Lemma \ref{same bucket} to Lemma \ref{same column} defined by 
        \begin{align*}
            C &\rightarrow A \\
            A &\rightarrow B \\
            B &\rightarrow D. 
        \end{align*}By Part 2 of Lemma \ref{same column}, $\angle CAB=135 \degree.$
    \item From Lemma \ref{same bucket} to Lemma \ref{same column}, map
        \begin{align*}
            A &\rightarrow C \\
            B &\rightarrow A \\
            D &\rightarrow B.
        \end{align*}By Part 1 of Lemma \ref{same column}, $\angle ABD=135 \degree$.
    \item From Lemma \ref{same bucket} to Lemma \ref{same row}, map
        \begin{align*}
            E &\rightarrow B \\
            A &\rightarrow A \\
            B &\rightarrow C \\
        \end{align*}
        From Part $1$ of Lemma \ref{same row}, $\angle EAB=135 \degree$.
    \item From Lemma \ref{same bucket} to  Lemma \ref{same row}, map
        \begin{align*}
            A &\rightarrow D \\
            B &\rightarrow B \\
            F &\rightarrow A
        \end{align*}
         By Part $2$ of Lemma \ref{same row}, $\angle ABF=135 \degree$.
    \end{enumerate}
\end{proof}

\begin{lemma} 
\label{outside columns} 
Let $n \in \mathbb{N}$ be arbitrary. Let $A$ and $B$ be distinct chosen points in an outside column, $k$, of $G_n$ with $A_y > B_y$.

\begin{enumerate}
    \item $k = 1$: If $B$ has a point, $D$, in the same slope bucket, $\angle{ABD} = 135 \degree$.
    \item $k = n$: If $A$ has a point, $C$ in the same slope bucket, $\angle{CAB} = 135 \degree$.
\end{enumerate}
\end{lemma}

\begin{figure}
    \centering
    \includegraphics[width=4.5cm]{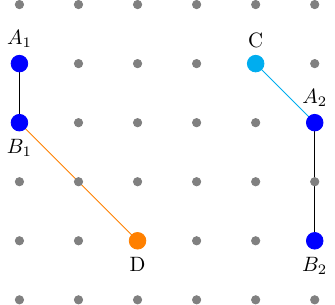}
    \caption{Lemma \ref{outside columns} $\angle ABD$ and $\angle CAB$}
    \label{fig: outside columns lemma}
\end{figure}

\begin{proof}
    Part 1 follows from Lemma \ref{same column} Case 2 and Part 2 follows from Lemma \ref{same column} Case 1. The angles are illustrated in Figure \ref{fig: outside columns lemma}.
\end{proof}

\begin{lemma}
\label{interiorBucket}
A point that is an interior point with respect to its slope bucket and has another chosen point in its row or column is part of a $135 \degree$ angle.
\end{lemma}

\begin{figure}[h]
\centering
\begin{minipage}{.5\textwidth}
  \centering
  \includegraphics[width=4.5cm]{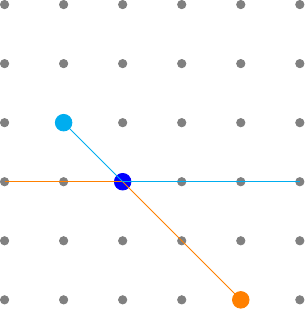}
  \captionof{figure}{Lemma \ref{interiorBucket} - Row}
  \label{fig:additional points in row}
\end{minipage}%
\begin{minipage}{.5\textwidth}
  \centering
  \includegraphics[width=4.5cm]{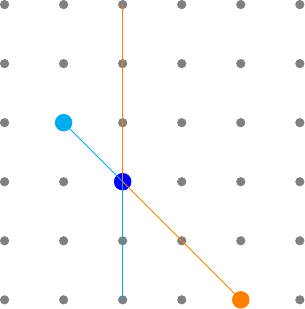} 
  \captionof{figure}{Lemma \ref{interiorBucket} - Column}
  \label{fig:additional points in column}
\end{minipage}
\end{figure}

\begin{proof}
    This lemma follows from Lemma \ref{same bucket}. The two possible cases are illustrated in Figure \ref{fig:additional points in row} and Figure \ref{fig:additional points in column}.
\end{proof}

\begin{lemma}
\label{interiorRow}
A point that is an interior point with respect to its row or column and has another chosen point in its slope bucket is part of a $135 \degree$ angle.
\end{lemma}

\begin{figure}[h]
\centering
\begin{minipage}{.5\textwidth}
  \centering
  \includegraphics[width=4.5cm]{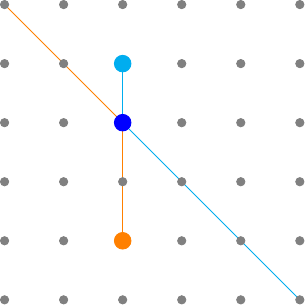}
  \captionof{figure}{Lemma \ref{interiorRow} - Column}
  \label{fig:interior point in column}
\end{minipage}%
\begin{minipage}{.5\textwidth}
  \centering
  \includegraphics[width=4.5cm]{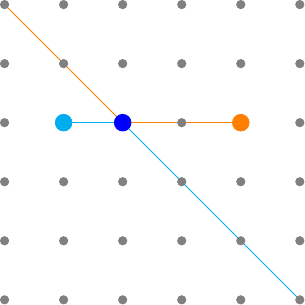}
  \captionof{figure}{Lemma \ref{interiorRow} - Row}
  \label{fig:interior point in row}
\end{minipage}
\end{figure}

\begin{proof}
    This lemma follows from Lemma \ref{same column} and Lemma \ref{same row}. The possible cases are shown in Figure \ref{fig:interior point in column} and Figure \ref{fig:interior point in row}.
\end{proof}

We have now arrived at the two most critical lemmas for proving our upper bound for angles of $135\degree$. 

\begin{lemma}
\label{strategy} Let $n \in \mathbb{N}$. Let $C_n$ be arbitrary such that $C_n$ contains $2n$ points and is peaceful. If $k$ is the number of slope buckets occupied by the points in $C_n$, at most $2n-1-k$ additional points can be added to $C_n$ without creating a $135 \degree$ angle.
\end{lemma}
\begin{proof}
Let $C_n$ be an arbitrary peaceful construction containing $2n$ points for $\theta = 135\degree$. Let $k$ be the number of slope buckets occupied by $C_n$. By the pigeonhole principle, for every additional point added to $C_n$, there exists some point in $C_n$ which is an interior point in its column. By Lemma \ref{interiorRow}, an interior point in its column must be the only point in its slope bucket. Since $G_n$ has $2n-1$ slope buckets and $k$ of them are occupied by $C_n$, at most $2n-1-k$ additional points can be added to $C_n$.
\end{proof}

\begin{lemma}
\label{klemma}
Let $C_n$ be an arbitrary peaceful construction containing $2n$ points for $\theta = 135\degree$. Let $k$ be the number of slope buckets occupied by $C_n$ points. Then
$ k \geq n+1$.
\end{lemma}

\begin{proof}
Let $C_n$ be an arbitrary peaceful construction containing $2n$ points for $\theta = 135\degree$. Let $k$ be the number of slope buckets that have at least one chosen point. Let $\mathcal{B}$ be the set of slope buckets containing at least $2$ points in $C_n$. Let $P \subseteq C_n$ be the points in $C_n$ belonging to a slope bucket of $\mathcal{B}$. Since $2n- \lvert P \rvert$ is the number of slope buckets containing one point in $C_n$ and $\lvert \mathcal{B} \rvert$ is the number of slope buckets containing $\geq 2$ points in $C_n$, by substitution,

\begin{equation}
   \label{eq2}
    k =2n- \lvert P \rvert +\lvert \mathcal{B} \rvert.
\end{equation}


Let $M$ be the set of interior points in slope buckets of $\mathcal{B}.$ Since each slope bucket contains a finite number of points in $G_n$, there are $2\lvert \mathcal{B} \rvert$ non-interior points in the slope buckets of $\mathcal{B}$. Therefore, 

\begin{equation}
\label{eq4}
    \lvert M \rvert = \lvert P \rvert - 2\lvert \mathcal{B} \rvert
\end{equation}

By Lemma \ref{interiorRow}, each column of the grid can contain at most two points of $P.$ Since there are $n$ columns on the grid, 

\begin{equation}
    \label{eq5}
    \lvert P \rvert \leq 2n.
\end{equation}

By Lemma \ref{outside columns}, the two exterior columns can contain at most $1$ point in a slope bucket of $\mathcal{B}$. Thus, 

 \begin{equation}
     \label{eq6}
     \lvert P \rvert \leq 2n-2.
 \end{equation}

Additionally, by Lemma \ref{interiorBucket}, any point in $M$ is the only point in its column. Therefore, 

    \begin{equation}
        \label{eq7}
        \begin{split}
    \lvert P \rvert &\leq 2n- \lvert M \rvert -2. \text{ Equivalently,}\\
       -\lvert P \rvert &\geq -(2n- \lvert M \rvert -2).  \text{ Adding 2n to both sides,}\\
       2n-\lvert P \rvert &\geq 2n-(2n- \lvert M \rvert -2)\\
       &= \lvert M \rvert +2.
       \end{split}
    \end{equation}

By substitution for $ \lvert M \rvert$ of Equation \ref{eq4} into Inequality \ref{eq7},

\begin{equation}
    \label{eq8}
    \begin{split}
        2n-\lvert P \rvert &\geq \lvert P \rvert-2\lvert \mathcal{B} \rvert +2\\
        2 \lvert \mathcal{B} \rvert &\geq 2\lvert P \rvert-2n+2\\
        \lvert \mathcal{B} \rvert &\geq \lvert P \rvert-n+1.
    \end{split}
\end{equation}


By substitution of Inequality \ref{eq8} into Equation \ref{eq2},

\begin{equation}
    \label{eq9}
    \begin{split}
        k &\geq 2n - \lvert P \rvert + (\lvert P \rvert-n+1) \\
        &= n+1
    \end{split}
\end{equation}
\end{proof}

\begin{theorem}
\label{upperbound135}
A peaceful construction $C_n$ can have at most $3n - 2$ points for $\theta = 135 \degree$. 
\end{theorem}

\begin{proof}
Let $C_n$ be a peaceful construction containing $2n$ points. By Lemma \ref{strategy}, we can add at most $2n-1-k$ additional points to $C_n$ for a total of $2n+2n-1-k=4n-1-k$ points. By Lemma \ref{klemma}, $k \geq n+1.$ Thus $C_n$ can contain at most $4n-1-(n+1)=3n-2$ points.
\end{proof}

\section{Upper Bound for all $\theta$}\label{secGeneralUpperBound}
We now generalize Theorem \ref{upperbound135} to prove an upper bound for all interesting angles $\theta$ such that $0 \degree<\theta<180 \degree.$ If $p, q < n$, let $f(p,q)$ be the number of $\poq$ slope buckets on $G_n$. 

\begin{lemma} We have the equality $f(p,q) = pn + qn - pq.$ \end{lemma}

\begin{proof}
    Let $n, p, q \in \mathbb{N}$ such that $p, q < n$. Let $\mathcal{S}$ be the set of points that are either in the first $p$ rows or the first $q$ columns of $G_n$. We now show that every point in $G_n$ is in the same slope bucket as one point in $\mathcal{S}.$ Let $B = (B_x, B_y) \in G_n$ be arbitrary. By the division algorithm, \[B_x = qs + B_x\text{ mod }q \andd B_y = pt + B_y\text{ mod }p\] for some $s \andd t \in \mathbb{N} \cup \{0\}$. Let

    $$ A = (A_x, A_y) =
    \begin{cases}
        (B_x-qt \text{, } B_y\text{ mod } p) & \text{if } t \leq s\\
        (B_x\text{ mod }q \text{, }B_y-ps) & \text{if } t>s.
    \end{cases}
    $$ It suffices to prove that the slope between $A$ and $B$ is $\poq$ and that $A_x \leq q$ or $A_y \leq p.$ Let $z_{A,B}$ be the slope between $A$ and $B$. If $t \leq s,$ $z_{A,B}$ is given by

    \begin{align}
        z_{A,B}&=\frac{B_y - A_y}{B_x-A_x} \nonumber\\
        &=\frac{(pt+ B_y \text{ mod } p)-(B_y \text{ mod } p)}{(qs+B_x \text{ mod } q)-(B_x-qt)}. \nonumber\\
        &\mbox{Substituting $B_x-qs$ for $B_x$ mod $q,$} \nonumber\\
        &=\frac{(pt+ B_y \text{ mod } p)-(B_y \text{ mod } p)}{(qs+B_x-qs)-(B_x-qt)}
        =\frac{pt}{qt}
        =\frac{p}{q}. \nonumber\\
        \nonumber
    \end{align} If $t>s,$ $z_{A,B}$ is given by

    \begin{align}
    z_{A,B}&=\frac{B_y-A_y}{B_x-A_x} \nonumber\\
    &=\frac{(pt+B_y \text{ mod } p)-(B_y-ps)}{(qs+B_x \text{ mod } q)-(B_x \text{ mod } q)}. \nonumber\\
    &\mbox{Substituting $B_y-pt$ for $B_y$ mod $p,$} \nonumber\\
        &=\frac{(pt+B_y-pt)-(B_y-ps)}{qs+B_x \text{ mod } q - (B_x \text{ mod } q)}
        =\frac{ps}{qs}
        =\frac{p}{q}. \nonumber\\
        \nonumber
    \end{align}

    In either case, $B$ falls into the $\poq$ slope bucket of $A.$ When $t \leq s,$ $A_y = B_y \text{ mod } p$ so $B_y <p.$ Thus, $A$ is in the first $p$ rows of the grid. When $t>s,$ $A_x=B_x \text{ mod }q$ so $A_x<q.$ Thus $A$ is in the first $q$ columns of the grid. Since $\lvert \mathcal{S} \rvert =pn +qn - pq$ and no points in $\mathcal{S}$ are in the same slope bucket, $f(p, q) = pn + qn - pq.$
\end{proof}

\begin{theorem}
    A construction $C_n$ that is peaceful for $\theta=\arctan(\frac{p}{q})$ has at most $2n + f(p,q) - 2(\max\left\{\lvert P \rvert , \lvert q \rvert \right\})$ points.
\end{theorem}

\begin{proof}
    Let $G_n$ be an arbitrary grid and $\theta$ an arbitrary interesting angle. By Theorem \ref{interestingIFFrational}, $\tan(\theta)=\poq$ for some integers $p$ and $q.$ Suppose that $C_n$ has $2n$ points and is peaceful for $\theta.$ The number of grid points in a $\poq-$slope-bucket is at most $\frac{n}{\text{max}\{ \lvert P \rvert ,\lvert q \rvert \}} .$ Then the $2n$ points must occupy at least $\frac{2n}{\left( \frac{n}{\text{max}\{\lvert p \rvert , \lvert q \rvert \}}\right)}=2 (\text{max}\{\lvert P \rvert , \lvert q \rvert \}).$
    
    By the pigeonhole principle, for every additional point chosen on the grid, there exists some  point which is an interior point in its column. A point which is an interior point in its column will be part of an angle of $\theta$ if there are any other points chosen in its slope bucket. Thus any additional points added to $C_n$ must be the only chosen points in their slope buckets for $C_n$ to remain peaceful for $\theta.$
    
    Since there are $f(p,q)$ slope buckets on the grid and at least $2 (\text{max}\{\lvert P \rvert , \lvert q \rvert \})$ have been used to place the first $2n$ points, at most $f(p,q)-2 (\text{max}\{\lvert P \rvert , \lvert q \rvert \})$ more points can be placed on the grid without forming an angle of $\theta.$ Thus at most $2n+f(p,q)-2 (\text{max}\{\lvert P \rvert , \lvert q \rvert \})$ points can be placed on the grid without forming an angle of $\theta=\tan^{-1}(\poq).$
\end{proof}

\section{Conclusion}\label{secConclusion}

When proving the upper bound for $135 \degree$ angles, we considered only angles that had one side in a row or column, which we called \emph{grid fitted}. However, other, \emph{sneaky} $135 \degree$ angles can be formed by grid points. An example is shown in Figure \ref{fig: sneaky}. The prevalence of sneaky angles makes finding configurations peaceful for $135 \degree$ a difficult task. For this reason, we believe the upper bound is not tight and more work remains to be done on this problem and many related problems. We end with a list of further questions about the combinatorial properties of objects on a grid.

\begin{figure}[h]
\centering
\includegraphics[width=4.5cm]{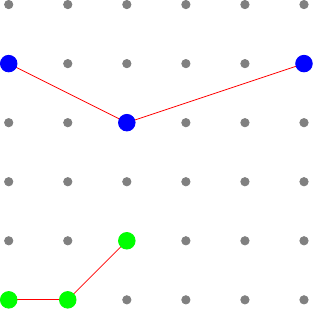} 
\caption{A grid fitted and a sneaky $135 \degree$ angle on $G_6.$}
\label{fig: sneaky}
\end{figure}

\begin{itemize}
    \item What other $135 \degree$ free constructions exist and what do they have in common?
    \item How does the number of possible sneaky angles relate to grid size?
    \item How may considering the specific geometry of angles improve their upper and lower bounds?
    \item What other shapes are interesting to consider? How many points can be placed on a grid such that no four make a square of any size (and rotation)?
    
\end{itemize}

\section{Statements and Declarations}
The work of each of the five authors was supported by National Science Foundation Grant 1852171.
The authors have no relevant financial or non-financial interests to disclose.

\bibliographystyle{amsplain}
\bibliography{no3bibliography}

\end{document}